\documentclass[a4paper,11pt]{amsart}
\renewcommand{\uppercasenonmath}[1]

\allowdisplaybreaks
\usepackage[T1]{fontenc}
\usepackage[utf8]{inputenc}

\usepackage{amsmath}
\usepackage{amsthm}
\usepackage{amssymb}
\usepackage{amstext}
\usepackage{mathtools}

\usepackage{mathrsfs}
\usepackage{bm}
\usepackage{shuffle}
\usepackage[all]{xy}
\usepackage{ytableau}
\ytableausetup{boxsize=normal,aligntableaux=center}

\usepackage{xcolor} 
\usepackage{graphicx}
\usepackage{tikz}

\usepackage{amsrefs}
\usepackage{ascmac}
\usepackage{comment}
\usepackage{here}
\usepackage{time}
\usepackage{fullpage} 

\usepackage{hyperref}
\hypersetup{
  colorlinks=true,
  linkcolor=blue,
  citecolor=blue,
  filecolor=blue,
  urlcolor=blue
}

\usepackage[capitalize,nameinlink,noabbrev,nosort]{cleveref}

\DeclareFontEncoding{OT2}{}{}
\DeclareFontSubstitution{OT2}{cmr}{m}{l}
\DeclareFontFamily{OT2}{cmr}{\hyphenchar\font45 }
\DeclareFontShape{OT2}{cmr}{m}{l}{%
  <5><6><7><8><9>gen*wncyr%
  <10><10.95><12><14.4><17.28><20.74><24.88>wncyr10}{}
\DeclareMathAlphabet{\mathcyr}{OT2}{cmr}{m}{l}
\DeclareMathAlphabet{\mathcyb}{OT2}{cmr}{b}{l}
\SetMathAlphabet{\mathcyr}{bold}{OT2}{cmr}{b}{l}

\makeatletter
\@namedef{subjclassname@2020}{%
  \textup{2020} Mathematics Subject Classification}
\makeatother

\newtheorem{theoremcounter}{Theorem Counter}[section]

\theoremstyle{definition}

\newtheorem{remark}[theoremcounter]{Remark}
\newtheorem{example}[theoremcounter]{Example}

\theoremstyle{plain}
\newtheorem{lemma}[theoremcounter]{Lemma}
\newtheorem{proposition}[theoremcounter]{Proposition}
\newtheorem{corollary}[theoremcounter]{Corollary}

\newtheorem{theorem}[theoremcounter]{Theorem}

\theoremstyle{definition}

\theoremstyle{plain}

\numberwithin{equation}{section}

\DeclareMathOperator{\ReNew}{Re}
\renewcommand{\Re}{\ReNew}

\definecolor{mizuiro}{RGB}{0,160,233}   
\definecolor{skyblue}{RGB}{86,180,233}  
\definecolor{darkgreen}{RGB}{0,128,0}   
\definecolor{navyblue}{RGB}{0,0,128}    

\begin{document}
\hypersetup{allcolors=black}

\title{Moment representations for generalized Lehmer polynomials}

\author{Karin Ikeda}
\address[Karin Ikeda]{Institute of Mathematics for Industry, Kyushu University, Motooka 744, Nishi-ku, Fukuoka 819-0395, Japan}
\email{ikeda@imi.kyushu-u.ac.jp}
\author{Hideki Murahara}
\address[Hideki Murahara]{The University of Kitakyushu,  4-2-1 Kitagata, Kokuraminami-ku, Kitakyushu, Fukuoka, 802-8577, Japan}
\email{hmurahara@mathformula.page}


\begin{abstract}
We represent normalized generalized Lehmer polynomials as moments of a random variable constructed from two independent negative binomial variables and an independent uniform variable. 
An integral form of their known exponential generating function yields mixtures of uniform distributions and convergent double series. 
We recover Eulerian polynomial identities and give a probabilistic interpretation of anti-diagonal sums of poly-Bernoulli numbers. 
We also prove moment determinacy, characterize the Stieltjes parameter range, derive Hankel total non-negativity and log-convexity, and identify the boundary distributions.
\end{abstract}

\maketitle

\section{Introduction}\label{sec:introduction}
\subsection{Central binomial series}
The central binomial series 
\[
\zeta_{CB}(s)\coloneqq\sum_{m=1}^{\infty}\frac{1}{\binom{2m}{m}m^s}
\]
converges absolutely and locally uniformly on $\mathbb C$, and hence defines an entire function.
Its values at non-positive integers lie in $\mathbb Q+\mathbb Q\pi/\sqrt3$. 
Lehmer \cite{Lehmer1985} introduced the polynomials $p_n(z)$ and $q_n(z)$ defined by the initial values $p_{-1}(z)=0$ and $q_{-1}(z)=1$ and the recurrences, valid for $n\ge-1$,
\begin{align}
q_{n+1}(z)
&=(2(n+1)z+1)q_n(z)+2z(1-z)q_n'(z),\label{eq:rec-q}\\
p_{n+1}(z)
&=2(nz+1)p_n(z)
+2z(1-z)p_n'(z)+q_n(z),\nonumber
\end{align}
and showed that, for every integer $k\ge1$,
\[
\zeta_{CB}(1-k)=\left(\frac23\right)^k
\left\{\frac12p_{k-1}\left(\frac14\right)
+\frac{\pi}{3\sqrt3}q_{k-1}\left(\frac14\right)\right\}.
\]
Borwein, Broadhurst, and Kamnitzer \cite{BorweinBroadhurstKamnitzer2001} showed, through log-sine integrals, that for every integer $k\ge2$ the value $\zeta_{CB}(k)$ is a $\mathbb Q$-linear combination of multiple zeta values, multiple Clausen values, and multiple Glaisher values. 
B\'enyi and Matsusaka \cite{BenyiMatsusaka2023}, Theorems~2.4 and~3.8, express Lehmer's polynomials in terms of bivariate Eulerian polynomials. 
Their Theorem~3.6 proves Stephan's observation
\begin{align}\label{eq:stephan}
\left(\frac23\right)^n p_n\left(\frac14\right)
=\sum_{k=0}^n B_{n-k}^{(-k)}\qquad(n\ge0),
\end{align}
where the poly-Bernoulli numbers $B_m^{(r)}$, $r\in\mathbb Z$, are defined by
\[
\sum_{m=0}^{\infty}B_m^{(r)}\frac{t^m}{m!}
=\sum_{j=1}^{\infty}\frac{(1-e^{-t})^{j-1}}{j^r}.
\]
Matsusaka \cite{Matsusaka2026} shows that, up to the factor $1/3$, these anti-diagonal sums are the values of a shifted log-sine integral at non-positive integers. 

\subsection{Generalized Lehmer polynomials}
Ikeda and Kadono \cite{IkedaKadono2026} introduce a Hurwitz--Lerch type central binomial series and polynomials $p_n(a,z)$ with $p_n(1,z)=p_n(z)$; we call them generalized Lehmer polynomials. Our variable $z$ corresponds to their polynomial variable $x$, which is the square of the third variable in their series. The initial value is $p_{-1}(a,z)=0$, and the defining recurrence, valid for $n\ge-1$, is
\begin{align}\label{eq:rec-p-general}
p_{n+1}(a,z)
=2((n+1-a)z+a)p_n(a,z)
+2z(1-z)\partial_zp_n(a,z)+q_n(z),
\end{align}
with $q_n(z)$ as above.
They obtain an Eulerian polynomial expansion and a recurrence for
$\alpha_n(a)\coloneqq (2/3)^np_n(a,1/4)$.
In this paper, we represent $p_n(a,z)/(n+1)$ and $\alpha_n(a)/(n+1)$ as moments; at $a=1$, this interprets the known anti-diagonal sums probabilistically.

\subsection{Moment representations}
Probabilistic representations of classical polynomials include Sun's complex moment formulas for Bernoulli and Euler polynomials and gamma moment formulas for Gegenbauer polynomials \cite{Sun2007}. 
Adell and Lekuona \cite{AdellLekuona2019} use uniform and exponential variables to represent Stirling numbers, while Adell and B\'enyi \cite{AdellBenyi2024} relate potential polynomials to moments of sums of independent variables. For Eulerian polynomials, shifted negative binomial moments appear in Charalambides \cite{Charalambides1991}; 
Barry \cite{Barry2018} studies moment representations and Hankel transforms using exponential Riordan arrays. 
Liang, Mu, and Wang \cite{LiangMuWang2016} use total positivity to obtain sufficient conditions for Catalan-like numbers to form Stieltjes moment sequences. A Stieltjes moment sequence is a real sequence $(m_n)_{n\ge0}$ with $m_n=\int_{[0,\infty)}x^n\,d\mu(x)$ for every $n\ge0$, for some finite positive Borel measure $\mu$ on $[0,\infty)$. 
Our representation combines two negative binomial variables with an independent uniform variable, and it gives a necessary and sufficient condition for the normalized generalized Lehmer polynomials to form a Stieltjes moment sequence.

\subsection{Main results}
Set
\[
\mathcal P(a,z;t)\coloneqq\sum_{n=0}^{\infty}p_n(a,z)\frac{t^{n+1}}{(n+1)!},
\qquad
\mathcal Q(z;t)
\coloneqq\sum_{n=0}^{\infty}q_{n-1}(z)\frac{t^n}{n!}.
\]
We regard $\mathcal P$ and $\mathcal Q$ primarily as formal series in $\mathbb Q[a,z][[t]]$, adjoining $(1-z)^{-1}$ when necessary; 
identities between polynomial coefficients then extend to $z=1$. 
For fixed real parameters with $z\ne1$, \cref{thm:generating} below identifies the series with analytic germs at $t=0$. 
Identities stated on larger intervals are understood through analytic continuation and do not assert Taylor convergence there.

We also fix our convention for negative binomial variables. 
For $y>0$ and $0\le z<1$, a negative binomial variable $N$ of shape $y$ and success probability $1-z$ is a random variable with
\[
\mathbb P(N=m)=\frac{(y)_m}{m!}(1-z)^yz^m
\qquad(m\ge0),
\]
where $(b)_0=1$ and $(b)_m=b(b+1)\cdots(b+m-1)$ for $m\ge1$. 
At $z=0$, this means that $N=0$ almost surely. 
The normalization by $n+1$ in the following theorem gives a moment sequence throughout the stated parameter range.

\begin{theorem}\label{thm:moment}
Let $a\in\mathbb R$ and $0\le z<1$. 
Let $N_1,N_2$ be independent negative binomial variables of shape $1/2$ and success probability $1-z$, and let $U$ be uniform on $[0,1]$ and independent of $N_1,N_2$. 
Define
\[
Y_{a,z}\coloneqq (1-z)\{2N_1+1+U(2N_2+2a-1)\}.
\]
Then, for every integer $n\ge0$,
\[
\mathbb E[Y_{a,z}^n]=\frac{p_n(a,z)}{n+1}.
\]
More precisely,
\[
\mathbb E[e^{tY_{a,z}}]=\frac{\mathcal P(a,z;t)}{t}
\quad\text{for }t\in I_z,
\qquad
I_z=\begin{cases}
\mathbb R,&z=0,\\
\displaystyle
\left(-\infty,\frac{\log(1/z)}{2(1-z)}\right),&0<z<1.
\end{cases}
\]
Here the quotient $\mathcal P(a,z;t)/t$ is defined at $t=0$ by continuity and has value $1$. On $I_z$ the function $\mathcal P(a,z;t)$ denotes the real analytic continuation of its germ at $t=0$, given by the integral formula in \cref{thm:generating} below.
\end{theorem}

We prove \cref{thm:moment} in \cref{sec:proof}, where we also give a mixture representation of the law of $Y_{a,z}$ and a double-series representation of $p_n(a,z)$. 
The following integral form of the known generating function, which we derive from the polynomial recurrences, provides a direct route to \cref{thm:moment}.

\begin{theorem}\label{thm:generating}
Let $a\in\mathbb R$ and $z\in\mathbb R\setminus\{1\}$. Define
\[
F_z(s)\coloneqq\left(\frac{1-z}{e^{(z-1)s}-z}\right)^{1/2}.
\]
For $t$ in a neighborhood of $0$, we have
\[
\mathcal Q(z;t)=F_z(2t),\qquad
\mathcal P(a,z;t)=\frac12F_z(2t)
\int_0^{2t}e^{(a-1)(1-z)s}F_z(s)\,ds.
\]
The square root is the branch with value $1$ at the origin.
\end{theorem}

We prove \cref{thm:generating} in \cref{sec:generating} and compare it there with the known hypergeometric formula.
Combining \cref{thm:moment} at $(a,z)=(1,1/4)$ with the identity \eqref{eq:stephan} of B\'enyi and Matsusaka gives the following interpretation.

\begin{corollary}\label{cor:polybernoulli}
Take $N_1,N_2,U$ as in \cref{thm:moment} with $(a,z)=(1,1/4)$.
Set
\[
Z_1\coloneqq N_1+\frac12+U\left(N_2+\frac12\right).
\]
Then $Z_1=(2/3)Y_{1,1/4}$.
For every integer $n\ge0$,
\[
\sum_{k=0}^n B_{n-k}^{(-k)}=(n+1)\mathbb E[Z_1^n].
\]
In particular, the sequence $\bigl((n+1)^{-1}\sum_{k=0}^n B_{n-k}^{(-k)}\bigr)_{n\ge0}$ is a Stieltjes moment sequence.
\end{corollary}

In \cref{sec:corollaries}, we obtain negative binomial representations at $a=0$ and a moment representation for Lehmer's $p_n(z)$ at $a=1$. 
We also derive the recurrence for $\alpha_n(a)$ from the differential equation in \cref{cor:ode}, and obtain the known Eulerian expansions from the random variable $Y_{a,z}$. 
For $0\le z<1$, we show that the normalized sequence $(p_n(a,z)/(n+1))_{n\ge0}$ is a Stieltjes moment sequence if and only if $a\ge0$, and in that case we derive Hankel total non-negativity and log-convexity. 
Finally, we give uniform and chi-square constructions at the boundary values $z=0$ and $z=1$, and a moment interpretation of the polynomial contribution to Hurwitz central binomial values.

\section{Preliminaries}\label{sec:preliminaries}
Throughout, $n$ is a non-negative integer unless stated otherwise, and $w \coloneqq 1-z$. 
Probability statements use $a\in\mathbb R$ and $0\le z<1$ unless another parameter range is specified. 
The first values of $p_n(a,z)$, computed from the recurrence \eqref{eq:rec-p-general}, are
\begin{align*}
p_0(a,z)&=1,\\
p_1(a,z)&=2(1-a)z+2a+1,\\
p_2(a,z)&=4(1-a)^2z^2-2(4a^2-3a-5)z+4a^2+2a+1.
\end{align*}
The values in \cref{tab:initial} agree with the table of Ikeda and Kadono \cite{IkedaKadono2026}.
\begin{table}[ht]
\centering
\begin{tabular}{c|c|c}
$n$ & $p_n(1,z)$ & $q_n(z)$\\ \hline
$-1$ & $0$ & $1$\\
$0$ & $1$ & $1$\\
$1$ & $3$ & $2z+1$\\
$2$ & $8z+7$ & $4z^2+10z+1$\\
$3$ & $20z^2+70z+15$ & $8z^3+60z^2+36z+1$
\end{tabular}
\vspace{3mm}
\caption{Initial values of Lehmer's polynomials.}\label{tab:initial}
\end{table}

For a negative binomial variable $N$ of shape $y>0$ and success probability $1-z$ with the convention fixed in \cref{sec:introduction}, the moment generating function is
\begin{align}\label{mom-neg}
\mathbb E[e^{sN}]=\left(\frac{1-z}{1-ze^s}\right)^y.
\end{align}
This formula holds for $s<\log(1/z)$ when $0<z<1$, and for every real $s$ when $z=0$.

For $z\ne1$, the bivariate Eulerian polynomials are defined by
\[
\mathcal E(z,y;s)
\coloneqq\sum_{n=0}^{\infty}E_n(z,y)\frac{s^n}{n!}
=\left(\frac{1-z}{e^{(z-1)s}-z}\right)^y.
\]
The power is the analytic branch with value $1$ at $s=0$. Each coefficient is a polynomial in $z,y$; at $z=1$, polynomial continuation gives $E_n(1,y)=(y)_n$. 
In particular, $F_z(s)=\mathcal E(z,1/2;s)$. 
The first bivariate Eulerian polynomials are
\[
E_0(z,y)=1,\quad E_1(z,y)=y,\quad 
E_2(z,y)=y^2+zy,\quad
E_3(z,y)=y^3+3zy^2+z^2y+zy.
\]
This family is related to the $1/k$-Eulerian polynomials $A_n^{(k)}(z)$ of Savage and Viswanathan \cite{SavageViswanathan2012}: 
by their equation~(6), the exponential generating function of $(A_n^{(k)}(z))_{n\ge0}$ is $\mathcal E(z,1/k;ks)$, so that $A_n^{(k)}(z)=k^nE_n(z,1/k)$ for positive integers $k$.

To study the moment sequences arising from normalized generalized Lehmer polynomials, we recall some standard notions from the moment problem.
A real sequence $(m_n)_{n\ge0}$ is a Hamburger moment sequence if $m_n=\int_{\mathbb R}x^n\,d\mu(x)$ for every $n\ge0$, for some finite positive Borel measure $\mu$ on $\mathbb R$ with moments of all orders; it is determinate if this measure is unique. 
It is a Stieltjes moment sequence if such a measure can be chosen on $[0,\infty)$, as in \cref{sec:introduction}. 
The Hankel matrix associated with $(m_n)_{n\ge0}$ is $(m_{i+j})_{i,j\ge0}$.
A matrix is totally non-negative if every finite minor is non-negative.

\section{Generating functions}\label{sec:generating}
\begin{lemma}\label{lem:pde}
For the differential operator
\[
D\coloneqq(1-2zt)\partial_t-2zw\partial_z,
\]
we have
\[
D\mathcal Q=\mathcal Q,\qquad
D\mathcal P=2aw\mathcal P+\mathcal Q.
\]
Together with $\mathcal Q(z;0)=1$ and $\mathcal P(a,z;0)=0$, these equations determine the formal series uniquely.
\end{lemma}

\begin{proof}
Multiply the recurrences \eqref{eq:rec-q} and \eqref{eq:rec-p-general} by $t^{n+1}/(n+1)!$ and sum over $n\ge-1$. The resulting equations are
\begin{align*}
\partial_t\mathcal Q&=2zt\partial_t\mathcal Q+\mathcal Q+2zw\partial_z\mathcal Q,\\
\partial_t\mathcal P&=2zt\partial_t\mathcal P+2aw\mathcal P+2zw\partial_z\mathcal P+\mathcal Q.
\end{align*}
Rearrangement gives the claims. Coefficient comparison recovers the original recurrences and hence proves uniqueness.
\end{proof}

\begin{proof}[Proof of \cref{thm:generating}]
Put
\[
Q(z,t)\coloneqq F_z(2t)=e^{wt}\left(\frac{w}{1-ze^{2wt}}\right)^{1/2}.
\]
Direct differentiation gives $DQ=Q$ and $Q(z,0)=1$, so \cref{lem:pde} gives $\mathcal Q=Q$. Next set
\[
c\coloneqq 2(a-1)w,\qquad g(z,u)\coloneqq e^{cu}Q(z,u),\qquad
J(z,t)\coloneqq \int_0^t g(z,u)\,du.
\]
Since $\partial_z c=-2(a-1)$, differentiation gives
\[
(1-2zu)\partial_u g-2zw\partial_z g=(1+c)g.
\]
Using $g(z,0)=1$ and integrating by parts, we obtain
\begin{align*}
DJ&=(1-2zt)g(z,t)
+\int_0^t\{-(1-2zu)\partial_u g+(1+c)g\}\,du\\
&=1+(1+c-2z)J=1+(2aw-1)J.
\end{align*}
Consequently, $D(QJ)=Q+2awQJ$ and $(QJ)(z,0)=0$. Since $\mathcal Q=Q$, the function $QJ$ satisfies the same equation and the same initial condition as $\mathcal P$ in \cref{lem:pde}, and the uniqueness in \cref{lem:pde} gives
\begin{align}\label{eq:P-integral}
\mathcal P(a,z;t)=Q(z,t)\int_0^t e^{2(a-1)wu}Q(z,u)\,du.
\end{align}
The substitution $s=2u$ proves the stated formula. The integrand is analytic near the origin, so the formal identities also hold for the corresponding analytic germs.
\end{proof}

\begin{corollary}\label{cor:ode}
For $z\ne1$, the function $\mathcal P$ satisfies
\[
(1-ze^{2wt})\partial_t\mathcal P
=w(\mathcal P+e^{2awt}),\qquad \mathcal P(a,z;0)=0.
\]
\end{corollary}
\begin{proof}
The identities $\partial_tQ/Q=w/(1-ze^{2wt})$ and $Q^2=we^{2wt}/(1-ze^{2wt})$, applied to \eqref{eq:P-integral}, give the result.
\end{proof}

\begin{remark}\label{rem:degenerateode}
The equation in \cref{cor:ode} is a formal identity also at $z=1$, where it reduces to $0=0$. At that parameter it does not determine $\mathcal P$ from its initial value.
\end{remark}

\begin{corollary}\label{cor:recurrence}
For every real $a,z$, $p_0(a,z)=1$ and, for $n\ge1$,
\[
p_n(a,z)=p_{n-1}(a,z)+(2aw)^n
+z\sum_{j=0}^{n-1}\binom nj2^{n-j}w^{n-j-1}p_j(a,z).
\]
\end{corollary}
\begin{proof}
For $z\ne1$, compare coefficients of $t^n/n!$ in \cref{cor:ode} and divide by $w$. The resulting identity is polynomial in $a,z$, so it holds also at $z=1$.
\end{proof}

\begin{proposition}\label{prop:beta}
Let $a\in\mathbb R$, $0<z<1$, $\tau\coloneqq(1-z)t$, and $X\coloneqq ze^{2\tau}$. For real $t$ near $0$,
\[
\mathcal P(a,z;t)=\frac{e^\tau}{2z^{a-1/2}\sqrt{1-X}}
\int_z^X u^{a-3/2}(1-u)^{-1/2}\,du.
\]
If $a>1/2$, the integral equals $B_X(a-1/2,1/2)-B_z(a-1/2,1/2)$, where
$B_x(\alpha,\beta)=\int_0^x u^{\alpha-1}(1-u)^{\beta-1}\,du$.
\end{proposition}
\begin{proof}
In \cref{thm:generating}, substitute $s=2v/w$ to obtain
\[
\mathcal P(a,z;t)
=\frac{e^\tau}{\sqrt{1-X}}
\int_0^\tau\frac{e^{(2a-1)v}}
{\sqrt{1-ze^{2v}}}\,dv.
\]
Then substituting $u=ze^{2v}$ and simplifying the prefactor gives the result.
\end{proof}

\begin{remark}\label{rem:hypergeometric}
For $a>1/2$, put
\[
H_a(x)\coloneqq{}_2F_1\left(\frac12,a-\frac12;a+\frac12;x\right).
\]
Since $B_x(a-1/2,1/2)=x^{a-1/2}H_a(x)/(a-1/2)$, \cref{prop:beta} gives
\[
\mathcal P(a,z;t)=\frac{e^\tau}{(2a-1)\sqrt{1-X}}
\{e^{(2a-1)\tau}H_a(X)-H_a(z)\}.
\]
This is the generating function obtained by Ikeda and Kadono \cite{IkedaKadono2026} in the proof of their Eulerian polynomial representation. Thus the generating function itself is already known; our integral form and its derivation make the probabilistic factorization explicit.

The definite integral also removes apparent singularities in the parameter. For $0<z,X<1$, its expansion, with $\alpha=a-1/2$, is
\[
\int_z^X u^{\alpha-1}(1-u)^{-1/2}\,du
=\sum_{k=0}^{\infty}\frac{(1/2)_k}{k!}
\frac{X^{\alpha+k}-z^{\alpha+k}}{\alpha+k}.
\]
At $\alpha=-k$, the quotient is $\log(X/z)$. 
The binomial series for $(1-u)^{-1/2}$ converges uniformly on the interval between $z$ and $X$, which justifies termwise integration for every complex $\alpha$. 
Each term of the resulting series is entire in $\alpha$, and the series converges locally uniformly in $\alpha$. 
Hence the integral is entire in $\alpha$, and any poles of the separate incomplete beta terms cancel.
\end{remark}

\section{Proof of the moment representation}\label{sec:proof}
\begin{proof}[Proof of \cref{thm:moment}] 
Write $X_i=w(2N_i+1)$ and $W=X_2+2(a-1)w$, so that $Y_{a,z}=X_1+UW$. By \eqref{mom-neg} with $y=1/2$ and $s=2wt$, we have
\[
\mathbb E[e^{tX_i}]=e^{wt}\left(\frac{w}{1-ze^{2wt}}\right)^{1/2}=F_z(2t).
\]
In \cref{thm:generating}, substitute $s=2tu$ to obtain
\begin{align*}
\frac{\mathcal P(a,z;t)}{t}
&=F_z(2t)\int_0^1e^{2(a-1)wtu}F_z(2tu)\,du\\
&=\mathbb E[e^{tX_1}]\int_0^1\mathbb E[e^{tuW}]\,du
=\mathbb E[e^{tY_{a,z}}].
\end{align*}
For $0<z<1$ and $t\in I_z$, both $ze^{2wt}<1$ and $ze^{2wtu}<1$ hold for every $u\in[0,1]$. The integrand is continuous and bounded on this interval. Tonelli's theorem and independence therefore justify the calculation, with finite exponential moments throughout $I_z$. At $z=0$, $Y_{a,0}$ is bounded. These formulas also give the asserted real analytic continuation.

Moreover,
\[
|Y_{a,z}|\le w(2N_1+2N_2+2|a|+2),
\]
so independence and \eqref{mom-neg} give, for $0<\varepsilon<\log(1/z)/(2w)$ when $0<z<1$ (and for any $\varepsilon>0$ when $z=0$),
\begin{align*}
\mathbb E[e^{\varepsilon|Y_{a,z}|}]
&\le e^{2\varepsilon w(|a|+1)}
\mathbb E[e^{2\varepsilon wN_1}]\mathbb E[e^{2\varepsilon wN_2}]\\
&=e^{2\varepsilon w(|a|+1)}\frac{w}{1-ze^{2\varepsilon w}}<\infty.
\end{align*}
Thus the law $\mu$ of $Y_{a,z}$ has a finite absolute exponential moment, meaning that
\[
\int_{\mathbb R}e^{\varepsilon|x|}\,d\mu(x)
=\mathbb E[e^{\varepsilon|Y_{a,z}|}]<\infty
\]
for some $\varepsilon>0$. Dominated convergence, using the bound $\sum_{n\ge0}|tY_{a,z}|^n/n!=e^{|t||Y_{a,z}|}\le e^{\varepsilon|Y_{a,z}|}$ for $|t|<\varepsilon$, allows us to expand the moment generating function and compare it with
\[
\frac{\mathcal P(a,z;t)}{t}
=\sum_{n=0}^{\infty}\frac{p_n(a,z)}{n+1}\frac{t^n}{n!}.
\]
This proves the moment identity, including its normalization at $t=0$.
\end{proof}

\begin{corollary}\label{cor:mixture}
Under the assumptions of \cref{thm:moment}, set
$\pi_m(z)=(1/2)_m(1-z)^{1/2}z^m/m!$. The law of $Y_{a,z}$ is
\[
\sum_{m,l\ge0}\pi_m(z)\pi_l(z)
\operatorname{Unif}\bigl(w(2m+1),\,2w(m+l+a)\bigr).
\]
Here $\operatorname{Unif}(b,c)$ denotes the uniform distribution on the interval with endpoints $b,c$, in either order, and denotes a point mass when $b=c$. Furthermore,
\[
p_n(a,z)=w^n\sum_{m,l\ge0}\pi_m(z)\pi_l(z)
\frac{\{2(m+l+a)\}^{n+1}-(2m+1)^{n+1}}{2l+2a-1},
\]
an absolutely convergent series. When the denominator vanishes, its quotient means $(n+1)(2m+1)^n$.
\end{corollary}
\begin{proof}
By independence, the event $N_1=m,N_2=l$ has probability $\pi_m(z)\pi_l(z)$, and on this event, the definition of $Y_{a,z}$ in \cref{thm:moment} gives $Y_{a,z}=w\{2m+1+U(2l+2a-1)\}$. 
Since $U$ is independent of $N_1,N_2$, the conditional law of $Y_{a,z}$ is the stated uniform distribution whenever $\pi_m(z)\pi_l(z)>0$.
If $b\coloneqq2m+1$ and $d\coloneqq2l+2a-1$, its $n$th moment is $w^n\int_0^1(b+ud)^n\,du$, and
\[
(n+1)\int_0^1(b+ud)^n\,du
=\begin{cases}
\bigl((b+d)^{n+1}-b^{n+1}\bigr)/d,&d\ne0,\\
(n+1)b^n,&d=0.
\end{cases}
\]
The absolute value is at most $(n+1)(2m+2l+2|a|+2)^n$. Negative binomial variables have moments of every order, which proves absolute convergence and justifies averaging.
\end{proof}

\begin{remark}\label{rem:normalization}
Since $\mathcal P(a,z;0)=0$, $\mathcal P$ itself is not a moment generating function. Its derivative satisfies
\[
\partial_t\mathcal P(a,z;t)=\frac{d}{dt}\bigl(t\mathbb E[e^{tY_{a,z}}]\bigr)
=\sum_{n=0}^{\infty}p_n(a,z)\frac{t^n}{n!}.
\]
The derivative need not be a moment generating function: at $(a,z)=(1,0)$, its proposed first two moments are $3$ and $7$, giving the negative variance $7-3^2=-2$. Indeed, $N_1=N_2=0$ almost surely at $z=0$, so $Y_{1,0}=1+U$, and \cref{thm:moment} gives $p_1(1,0)=2\mathbb E[1+U]=3$ and $p_2(1,0)=3\mathbb E[(1+U)^2]=7$. At $a=0$, however, it is a moment generating function, as shown below.
\end{remark}

\section{Corollaries and examples}\label{sec:corollaries}
\begin{corollary}\label{cor:negativebinomial}
For $0\le z<1$ and $n\ge0$, we have $p_n(0,z)=q_n(z)$ and
\[
q_n(z)=\mathbb E\bigl[\{(1-z)(2\widetilde N+1)\}^n\bigr],
\]
where $\widetilde N$ has negative binomial shape $3/2$ and success probability $1-z$. Moreover, if $y>0$ and $N$ has shape $y$ and the same success probability, then
\[
E_n(z,y)=(1-z)^n\mathbb E[(N+y)^n].
\]
\end{corollary}
\begin{proof}
For $a=0$, $\mathcal Q-1$ satisfies the equation and initial condition in \cref{cor:ode}. Uniqueness gives $\mathcal P(0,z;t)=\mathcal Q(z;t)-1$ near $0$. Differentiating, we obtain
\[
\partial_t\mathcal Q(z;t)
=e^{wt}\left(\frac{w}{1-ze^{2wt}}\right)^{3/2},
\]
which, by \eqref{mom-neg} with $y=3/2$ and $s=2wt$, is the moment generating function in the first assertion. 
For the second, \eqref{mom-neg} with $ws$ in place of $s$ gives
\[
\mathbb E[e^{sw(N+y)}]
=e^{yws}\left(\frac{w}{1-ze^{ws}}\right)^y
=\mathcal E(z,y;s).
\]
Coefficient comparison proves both identities.
\end{proof}

\begin{remark}\label{rem:charalambides}
The identity for general $y$ is already contained in the series expansion of Savage and Viswanathan \cite{SavageViswanathan2012}, equation~(9). 
It also follows from Charalambides \cite{Charalambides1991}, equation~(3.4), by setting $\alpha=0,\beta=y$ in his polynomial $A_n(z;\alpha,\beta)$; his positive-parameter formula extends continuously to $\alpha=0$. 
The negative binomial interpretation is therefore already known. 
The identity $p_n(0,z)=q_n(z)$ is also noted by Ikeda and Kadono \cite{IkedaKadono2026}; it does not extend to $n=-1$.
\end{remark}

\begin{corollary}\label{cor:lehmer}
For $0\le z<1$, with $N_1,N_2,U$ as in \cref{thm:moment}, Lehmer's polynomials satisfy
\[
\frac{p_n(z)}{n+1}
=\mathbb E\bigl[\{(1-z)(2N_1+1+U(2N_2+1))\}^n\bigr].
\]
\end{corollary}
\begin{proof}
Set $a=1$ in \cref{thm:moment}.
\end{proof}

\begin{corollary}\label{cor:alpha}
For $a\in\mathbb R$, take $N_1,N_2,U$ as in \cref{thm:moment} with $z=1/4$.
Put
\[
Z_a\coloneqq N_1+\frac12+U\left(N_2+a-\frac12\right),
\qquad \alpha_n(a)\coloneqq\left(\frac23\right)^np_n\left(a,\frac14\right).
\]
Then $\alpha_n(a)=(n+1)\mathbb E[Z_a^n]$. Moreover, $\alpha_0(a)=1$ and
\[
3\alpha_n(a)=2\alpha_{n-1}(a)
+\sum_{j=0}^{n-1}\binom nj\alpha_j(a)+3a^n\qquad(n\ge1).
\]
The first values are
\[
\alpha_1(a)=a+1,\qquad
\alpha_2(a)=a^2+\frac{4a+5}{3},\qquad
\alpha_3(a)=a^3+\frac{5a^2}{3}+\frac{29a+37}{9}.
\]
\end{corollary}
\begin{proof}
Since $Z_a=(2/3)Y_{a,1/4}$, \cref{thm:moment} gives the moments. Set
\[
G_a(t)\coloneqq t\mathbb E[e^{tZ_a}]
=\frac32\mathcal P\left(a,\frac14;\frac{2t}{3}\right).
\]
Rescaling \cref{cor:ode} gives $(4-e^t)G_a'(t)-2G_a(t)=3e^{at}$, while
$G_a'(t)=\sum_{n\ge0}\alpha_n(a)t^n/n!$. Comparing coefficients recovers the recurrence of Ikeda and Kadono \cite{IkedaKadono2026} and gives the listed values.
\end{proof}

\begin{proof}[Proof of \cref{cor:polybernoulli}]
Combine \cref{cor:alpha} at $a=1$ with the identity \eqref{eq:stephan} of B\'enyi and Matsusaka \cite{BenyiMatsusaka2023}, Theorem~3.6.
Since $Z_1\ge0$, the normalized sequence is Stieltjes. For $n=0,1,2,3$, the sums are $1,2,4,10$, and division by $n+1$ gives the corresponding moments $\mathbb E[Z_1^n]$, namely $1,1,4/3,5/2$.
\end{proof}

\begin{corollary}\label{cor:eulerian}
For all real $a,z$ and $n\ge0$,
\[
p_n(a,z)=2^n\sum_{j=0}^n\sum_{l=0}^j
\binom{n+1}{j+1}\binom jl(a-1)^{j-l}(1-z)^{j-l}
E_{n-j}\left(z,\frac12\right)E_l\left(z,\frac12\right).
\]
In particular,
\[
p_n(z)=2^n\sum_{k=0}^n\binom{n+1}{k}
E_{n-k}\left(z,\frac12\right)E_k\left(z,\frac12\right).
\]
\end{corollary}
\begin{proof}
First suppose $0\le z<1$ and use $Y_{a,z}=X_1+UW$ from \cref{sec:proof}. Independence and $\mathbb E[U^j]=1/(j+1)$ give
\[
(n+1)\mathbb E[Y_{a,z}^n]
=\sum_{j=0}^n\binom{n+1}{j+1}\mathbb E[X_1^{n-j}]\mathbb E[W^j].
\]

Now \cref{cor:negativebinomial} with $y=1/2$ gives $\mathbb E[X_i^r]=2^rE_r(z,1/2)$, and
\[
\mathbb E[W^j]=2^j\sum_{l=0}^j\binom jl(a-1)^{j-l}w^{j-l}E_l(z,1/2).
\]
This proves the first identity, which extends to every real $z$ because both sides are polynomials. At $a=1$, only $l=j$ remains; replacing $j$ by $n-k$ gives the second identity. These are the formulas of Ikeda and Kadono \cite{IkedaKadono2026} and B\'enyi and Matsusaka \cite{BenyiMatsusaka2023}, Theorem~3.8.
\end{proof}

\begin{corollary}\label{cor:stieltjes}
Fix $a\in\mathbb R$ and $0\le z<1$, and set $m_n=p_n(a,z)/(n+1)$. The sequence $(m_n)_{n\ge0}$ is a determinate Hamburger moment sequence (see \cref{sec:preliminaries}). 
It is a Stieltjes moment sequence if and only if $a\ge0$. In that case the Hankel matrix $(m_{i+j})_{i,j\ge0}$ has all minors non-negative and
\[
m_n^2\le m_{n-1}m_{n+1}\qquad(n\ge1).
\]
\end{corollary}
\begin{proof}
The law $\mu$ of $Y_{a,z}$ represents the moments and, by the estimate in the proof of \cref{thm:moment}, satisfies $\int_{\mathbb R}e^{\varepsilon|x|}\,d\mu(x)<\infty$ for some $\varepsilon>0$.
 To see uniqueness, let $\nu$ be any other finite positive Borel measure on $\mathbb R$ with the same moments. For this $\varepsilon$, Tonelli's theorem gives
\[
\int\cosh(\varepsilon x)\,d\nu(x)
=\sum_{n=0}^{\infty}\frac{\varepsilon^{2n}m_{2n}}{(2n)!}
=\int\cosh(\varepsilon x)\,d\mu(x)<\infty.
\]
Since $e^{\varepsilon|x|}\le2\cosh(\varepsilon x)$, both measures have finite absolute exponential moments. Their moment generating functions are analytic on $|\Re (t)|<\varepsilon$ and have the same Taylor coefficients at $0$. By the identity theorem, these functions coincide throughout the strip and hence on the imaginary axis. Their characteristic functions therefore coincide, and uniqueness of characteristic functions implies $\mu=\nu$.

The mixture in \cref{cor:mixture} is supported on $[0,\infty)$ when $a\ge0$. If $a<0$, the event $N_1=N_2=0$ has probability $w$, and on it $Y_{a,z}=w\{1+(2a-1)U\}$. Hence
\[
\mathbb P(Y_{a,z}<0)\ge w\frac{-2a}{1-2a}>0.
\]
Determinacy excludes any alternative representing measure on $[0,\infty)$. The Stieltjes characterization recalled by Liang, Mu, and Wang \cite{LiangMuWang2016}, Section~1, gives total non-negativity of the Hankel matrix (called total positivity there). 
The log-convexity inequality follows from Cauchy--Schwarz applied to $Y_{a,z}^{(n-1)/2}$ and $Y_{a,z}^{(n+1)/2}$.
\end{proof}

\begin{example}\label{ex:variance}
With $m_n=p_n(a,z)/(n+1)$ and $h\coloneqq1+2(a-1)(1-z)$, the normalization, mean, and variance of $Y_{a,z}$ are given by
\[
m_0=1,\qquad m_1=1+\frac{h} {2},\qquad
m_2-m_1^2=\frac{8z}{3}+\frac{h^2}{12}.
\]
Indeed, with $X_i$ and $W$ as in \cref{sec:proof}, $\mathbb E[X_i]=1$, $\operatorname{Var}(X_i)=2z$, $\mathbb E[W]=h$, and independence gives
$\operatorname{Var}(X_1+UW)=2z+(2z+h^2)/3-h^2/4$.
Thus the first $2\times2$ Hankel determinant $m_0m_2-m_1^2=m_2-m_1^2$ is non-negative for every real $a$, including parameters outside the Stieltjes range.
\end{example}

\begin{example}\label{ex:uniform}
We specialize \cref{thm:moment} to the boundary value $z=0$; the other boundary value, $z=1$, is treated in \cref{cor:boundary}.
At $z=0$, $N_1=N_2=0$ almost surely, so $Y_{a,0}=1+(2a-1)U$ is uniform on the interval with endpoints $1,2a$, with a point mass at $1$ when $a=1/2$. Thus
\[
\frac{\mathcal P(a,0;t)}{t}
=\begin{cases}
\displaystyle\frac{e^{2at}-e^t}{(2a-1)t},&a\ne1/2,\\
e^t,&a=1/2,
\end{cases}
\qquad
p_n(a,0)=\sum_{j=0}^n(2a)^j.
\]
These formulas follow by integrating $e^{t(1+(2a-1)u)}$ and $(1+(2a-1)u)^n$ on $[0,1]$ and using $p_n(a,0)=(n+1)\mathbb E[Y_{a,0}^n]$. The quotient at $t=0$ is interpreted by continuity.
\end{example}

\begin{corollary}\label{cor:boundary}
For every real $a$, $p_n(a,1)=(2n+1)!!$. 
Let $V_1,V_2$ be independent chi-square variables with one degree of freedom, and let $U$ be an independent uniform variable on $[0,1]$. 
Then $\widehat Y=V_1+UV_2$ satisfies
\[
\mathbb E[e^{t\widehat Y}]=\frac{(1-2t)^{-1/2}-1}{t}\quad(t<1/2),
\qquad
\mathbb E[\widehat Y^n]=\frac{(2n+1)!!}{n+1},
\]
where the quotient at $t=0$ is $1$. 
For fixed $a$, $Y_{a,z}$ converges in distribution to $\widehat Y$ as $z\to1^-$, and all its moments converge. 
In particular, the boundary sequence is Stieltjes for every real $a$.
\end{corollary}
\begin{proof}
At $z=1$, the term $(2aw)^n$ vanishes and only the summand with $j=n-1$ survives, so \cref{cor:recurrence} gives $p_n(a,1)=(2n+1)p_{n-1}(a,1)$. 
Thus $\mathcal P(a,1;t)=(1-2t)^{-1/2}-1$ near $0$. 
A chi-square variable with one degree of freedom has gamma shape $1/2$, scale $2$, and moment generating function $(1-2t)^{-1/2}$. 
Consequently,
\[
\mathbb E[e^{t\widehat Y}]=(1-2t)^{-1/2}\int_0^1(1-2tu)^{-1/2}\,du
=\frac{(1-2t)^{-1/2}-1}{t}. 
\]
As $z\to1^-$, the moment generating functions of $X_i=(1-z)(2N_i+1)$ converge to $(1-2t)^{-1/2}$ on a neighborhood of $0$. 
Independence yields joint convergence with $U$, and $2(a-1)(1-z)\to0$. 
The continuous mapping theorem gives the distributional limit. Moment convergence follows separately from the polynomial identity $\mathbb E[Y_{a,z}^n]=p_n(a,z)/(n+1)$ and the displayed limiting moments.
\end{proof}

\begin{remark}\label{rem:hurwitzvalues}
For $a>1/2$, define
\[
\zeta_{HCB}(s,a)\coloneqq\sum_{n=0}^{\infty}
\frac{1}{\binom{2(n+a)}{n+a}(n+a)^s}, 
\qquad
\binom{2v}{v}\coloneqq\frac{\Gamma(2v+1)}{\Gamma(v+1)^2}.
\]
The incomplete beta formula of Ikeda and Kadono \cite{IkedaKadono2026} gives, for integers $k\ge1$,
\[
\zeta_{HCB}(1-k,a)
=C_a\left\{\frac{2k}{3}\mathbb E[Z_a^{k-1}]
+\frac{4^a}{2\sqrt3}B_{1/4}\left(a-\frac12,\frac12\right)\mathbb E[T^k]\right\},
\]
where $C_a=(2a-1)\Gamma(a+1)^2/(a\Gamma(2a+1))$ and
$T=N_1+1/2$, with $N_1$ as in \cref{cor:alpha}.
Indeed, in their formula substitute
\[
\left(\frac23\right)^kp_{k-1}\left(a,\frac14\right)
=\frac{2k}{3}\mathbb E[Z_a^{k-1}],
\qquad
\left(\frac23\right)^kq_{k-1}\left(\frac14\right)=\mathbb E[T^k].
\]
The first substitution is \cref{cor:alpha}, and the second follows from $\mathbb E[e^{tX_1}]=F_z(2t)=\mathcal Q(z;t)$ in \cref{sec:proof} and \cref{thm:generating}, since $X_1=(3/2)T$ at $z=1/4$. 
This identifies the polynomial contribution as a scaled moment, without asserting rationality for general $a$. 
For a positive integer $a=m$, the shift identity for $\zeta_{HCB}$ in the same paper gives its rational part in the decomposition into $\mathbb Q+\mathbb Q\pi/\sqrt3$ as
\[
\frac{k}{3}\mathbb E[Z_1^{k-1}]
-\sum_{j=1}^{m-1}\frac{j^{k-1}}{\binom{2j}{j}}.
\]
Indeed, $\zeta_{HCB}(1-k,m)=\zeta_{CB}(1-k)-\sum_{j=1}^{m-1}j^{k-1}/\binom{2j}{j}$. 
For $m>1$, the incomplete beta term can also contribute to the rational part.
\end{remark}

\section*{Acknowledgements}
The authors thank Professor Fumio Hiroshima for inviting the first-named author to speak at the workshop ``Zeta Functions and Non-commutativity'' at Kyushu University. 
This work originated from that talk.
The first author is supported by JSPS KAKENHI Grant Number JP26K24491. 

\section*{Declaration on the use of generative AI}
During the preparation of this manuscript, the authors used the generative AI tools ChatGPT (OpenAI) and Claude (Anthropic) for language editing, improving the clarity of the presentation, and assisting in checking mathematical arguments. 
All mathematical arguments, statements, and conclusions were independently verified by the authors, who take full responsibility for the content of the manuscript.

\bibliographystyle{amsplain}
\bibliography{References}

\begin{bibdiv}
\begin{biblist}

\bib{AdellBenyi2024}{article}{
      author={Adell, Jos{\'e}~A.},
      author={B{\'e}nyi, Be{\'a}ta},
       title={Probabilistic {S}tirling numbers and applications},
        date={2024},
     journal={Aequationes Mathematicae},
      volume={98},
      number={6},
       pages={1627\ndash 1646},
}

\bib{AdellLekuona2019}{article}{
      author={Adell, Jos{\'e}~A.},
      author={Lekuona, Alberto},
       title={Explicit expressions and integral representations for the {S}tirling numbers. {A} probabilistic approach},
        date={2019},
     journal={Advances in Difference Equations},
      volume={2019},
       pages={Paper No. 398},
}

\bib{Barry2018}{misc}{
      author={Barry, Paul},
       title={The {$1/k$}-{E}ulerian polynomials as moments, via exponential {R}iordan arrays},
        date={2018},
        note={\href{https://arxiv.org/abs/1803.10308}{arXiv:1803.10308}},
}

\bib{BenyiMatsusaka2023}{article}{
      author={B{\'e}nyi, Be{\'a}ta},
      author={Matsusaka, Toshiki},
       title={Remarkable relations between the central binomial series, {E}ulerian polynomials, and poly-{B}ernoulli numbers, leading to {S}tephan's observation},
        date={2023},
     journal={Kyushu Journal of Mathematics},
      volume={77},
      number={1},
       pages={149\ndash 158},
      eprint={2207.00205},
}

\bib{BorweinBroadhurstKamnitzer2001}{article}{
      author={Borwein, Jonathan~M.},
      author={Broadhurst, David~J.},
      author={Kamnitzer, Joel},
       title={Central binomial sums, multiple {C}lausen values, and zeta values},
        date={2001},
     journal={Experimental Mathematics},
      volume={10},
      number={1},
       pages={25\ndash 34},
}

\bib{Charalambides1991}{article}{
      author={Charalambides, Charalambos~A.},
       title={On a generalized {E}ulerian distribution},
        date={1991},
     journal={Annals of the Institute of Statistical Mathematics},
      volume={43},
      number={1},
       pages={197\ndash 206},
}

\bib{IkedaKadono2026}{article}{
      author={Ikeda, Karin},
      author={Kadono, Yuta},
       title={Hurwitz--{L}erch type central binomial series},
        date={2026},
     journal={The Ramanujan Journal},
      volume={69},
      number={3},
       pages={Paper No. 61},
      eprint={2409.16026},
}

\bib{Lehmer1985}{article}{
      author={Lehmer, D.~H.},
       title={Interesting series involving the central binomial coefficient},
        date={1985},
     journal={The American Mathematical Monthly},
      volume={92},
      number={7},
       pages={449\ndash 457},
}

\bib{LiangMuWang2016}{article}{
      author={Liang, Huyile},
      author={Mu, Lili},
      author={Wang, Yi},
       title={Catalan-like numbers and {S}tieltjes moment sequences},
        date={2016},
     journal={Discrete Mathematics},
      volume={339},
      number={2},
       pages={484\ndash 488},
}

\bib{Matsusaka2026}{misc}{
      author={Matsusaka, Toshiki},
       title={Poly-{B}ernoulli numbers from shifted log-sine integrals},
        date={2026},
        note={\href{https://arxiv.org/abs/2603.25392}{arXiv:2603.25392}},
}

\bib{SavageViswanathan2012}{article}{
      author={Savage, Carla~D.},
      author={Viswanathan, Gopal},
       title={The {$1/k$}-{E}ulerian polynomials},
        date={2012},
     journal={The Electronic Journal of Combinatorics},
      volume={19},
      number={1},
       pages={Paper No. P9},
}

\bib{Sun2007}{article}{
      author={Sun, Ping},
       title={Moment representation of {B}ernoulli polynomial, {E}uler polynomial and {G}egenbauer polynomials},
        date={2007},
     journal={Statistics \& Probability Letters},
      volume={77},
      number={7},
       pages={748\ndash 751},
}

\end{biblist}
\end{bibdiv}

\end{document}